\documentclass[10pt,letter]{article} 

\usepackage{tikz}
\tikzset{every picture/.style={line width=0.75pt}} 

 \usepackage[usenames,dvipsnames]{pstricks}
 \usepackage{epsfig}
 \usepackage{pst-grad} 
 \usepackage{pst-plot} 
 \usepackage[space]{grffile} 
 \usepackage{etoolbox} 
 \makeatletter 
 \patchcmd\Gread@eps{\@inputcheck#1 }{\@inputcheck"#1"\relax}{}{}
 \makeatother

\usepackage{dsfont}

\usepackage[T1]{fontenc} 

\usepackage{amsmath,mathtools}
\usepackage{amssymb}
\usepackage{amsthm}
\usepackage[latin1]{inputenc}
     
\usepackage{graphicx}

\usepackage[usenames,dvipsnames]{pstricks}
\usepackage{epsfig}
\usepackage{pst-grad} 
\usepackage{pst-plot} 

\newcommand\restr[2]{{
  \left.\kern-\nulldelimiterspace 
  #1 
  \vphantom{\Big|} 
  \right|_{#2} 
  }}

\usepackage[misc,geometry]{ifsym}

\usepackage{ifthen}
\usepackage{color}

\newcommand{\intav}[1]{\mathchoice {\mathop{\vrule width 6pt height 3 pt depth  -2.5pt
\kern -8pt \intop}\nolimits_{\kern -6pt#1}} {\mathop{\vrule width
5pt height 3  pt depth -2.6pt \kern -6pt \intop}\nolimits_{#1}}
{\mathop{\vrule width 5pt height 3 pt depth -2.6pt \kern -6pt
\intop}\nolimits_{#1}} {\mathop{\vrule width 5pt height 3 pt depth
-2.6pt \kern -6pt \intop}\nolimits_{#1}}}

\def\polhk#1{\setbox0=\hbox{#1}{\ooalign{\hidewidth\lower1.5ex\hbox{`}\hidewidth\crcr\unhbox0}}}

\def\XXint#1#2#3{{\setbox0=\hbox{$#1{#2#3}{\int}$ }
\vcenter{\hbox{$#2#3$ }}\kern-.6\wd0}}

\usepackage{setspace}

\newtheorem{theorem}{Theorem}
\newtheorem{example}{Example}
\newtheorem{definition}{Definition}

\newtheorem{corollary}{Corollary}
\newtheorem{proposition}{Proposition}
\newtheorem{remark}{Remark}
\newtheorem{assumption}{A}

\makeatletter
\patchcmd{\env@cases}{1.2}{1}{}{}
\makeatother

\usepackage{graphicx}
\usepackage{subcaption}
\usepackage{algorithm}
\usepackage{algpseudocode}

\begin{document}

\title{A numerical scheme for a fully nonlinear free boundary problem}

\author{Edgard A. Pimentel and Erc\'ilia Sousa}
	
\date{\today} 

\maketitle

\begin{abstract}

\noindent We propose a numerical method to approximate viscosity solutions of fully nonlinear free transmission problems. The method discretises a two-layer regularisation of a PDE, involving a functional and a vanishing parameter. The former is handled via a fixed-point argument. We then prove that the numerical method converges to a one-parameter regularisation of the free boundary problem. Regularity estimates enable us to take the vanishing limit of such a parameter and recover a viscosity solution of the free transmission problem. Our main contribution is the design of a computational strategy, based on fixed-point arguments and approximated problems, to solve fully nonlinear free boundary models. We finish the paper with two numerical examples to validate our method. 

\smallskip


\smallskip 

\end{abstract}

\vspace{.1in}

\section{Introduction}\label{sec_intro}

We propose a finite difference method for the free boundary problem
\begin{equation}\label{eq_main1}
	\begin{cases}
		F_1(D^2u)=f&\hspace{.2in}\mbox{in}\hspace{.1in}\Omega\cap\left\lbrace u>0\right\rbrace\\
		F_2(D^2u)=f&\hspace{.2in}\mbox{in}\hspace{.1in}\Omega\cap\left\lbrace u<0\right\rbrace\\
		u=g&\hspace{.2in}\mbox{on}\hspace{.1in}\partial\Omega,
	\end{cases}
\end{equation}
where $F_1,F_2:S(d)\to\mathbb{R}$ are uniformly elliptic operators, $f\in L^\infty(\Omega)\cap C(\Omega)$, and $g\in C(\partial\Omega)$. The domain $\Omega\subset\mathbb{R}^d$ is an open and bounded set satisfying a uniform exterior cone condition. 

The free boundary problem \eqref{eq_main1} describes a discontinuous diffusion process, with solution-dependent discontinuities taking place at the level of the operator. This model falls within the scope of the so-called \emph{free transmission problems} and was introduced in \cite{PimSwi,PimSan2023}. 

Transmission problems were introduced in the mid-1950s by Mauro Picone \cite{Picone} and have attracted considerable attention from the mathematical community since then. Firstly, this class of problems was studied in the context of fixed interfaces, where discontinuities occur across a priori known subregions. In recent years, transmission models have been formulated in the context of free boundaries. For a comprehensive review of the subject, we refer the reader to \cite{BiaPimUrb}. 

Fundamental questions concerning the existence of solutions to \eqref{eq_main1} and their regularity properties have been studied in the literature. In \cite{PimSwi}, the authors establish the existence of viscosity solutions to \eqref{eq_main1}. They suppose the operators $F_1$ and $F_2$ are $(\lambda,\Lambda)$-elliptic, $g\in C(\partial\Omega)$, and $f\in L^p(\Omega)$, where $d/2<p_0<p$. Here, $p_0$ is the integrability level above which the Aleksandroff-Bakelman-Pucci estimate is available for the viscosity solutions to $F(D^2u)=f$. Under a near-convexity condition for $F_1$ and $F_2$, the authors prove the existence of strong solutions to \eqref{eq_main1}.

The regularity of the solutions to \eqref{eq_main1} is the subject of \cite{PimSan2023}. In that paper, the authors prove that viscosity solutions are in $W^{2,{\rm BMO}}_{\rm loc}(\Omega)$, with estimates. Under a density condition on the negative phase, they establish (optimal) $C^{1,1}_{\rm loc}$-regularity estimates. Regularity estimates at the intersection of free and fixed boundaries are the subject of \cite{JesPimSto}.

The computational aspects of \eqref{eq_main1} remain largely open. Though the numerical analysis of viscosity solutions has been extensively studied in the literature (see \cite{BarSou,BonZid,Bonnet,CagGomTra,CapFal,CraLio1984,DebJak,Fal,FenJen,Obe,Sou}, to mention just a few), genuine difficulties stem from the model under analysis. A fundamental difficulty in the numerical treatment of \eqref{eq_main1} concerns the solution-dependent discontinuities of the operator driving the diffusion process. Among other things, one cannot write \eqref{eq_main1} as an equality driven by the maximum operator -- as usual in the numerical analysis of the obstacle problem, for example; see \cite{Obe}. Our main contribution is in formulating computationally the strategy put forward in \cite{PimSwi}.

In brief, the argument in \cite{PimSwi} develops as follows. For $0<\varepsilon\ll1$ and $v\in C(\overline\Omega)$ fixed, define $h_\varepsilon^v:\Omega\to[0,1]$ as
\[
	h_\varepsilon^v(x)\coloneqq\left[\max\left(\min\left(\frac{v(x)+\varepsilon}{2\varepsilon},1\right),0\right)\ast\eta_\varepsilon\right](x),
\]
where $\eta_\varepsilon$ is a standard symmetric mollifying kernel. Define the fully nonlinear elliptic operator $F_\varepsilon^v:S(d)\times\mathbb{R}\times\Omega\to\mathbb{R}$ as
\begin{equation}\label{eq_nm0}
	F_{\varepsilon}^v(M,r,x)\coloneqq\varepsilon r+h_\varepsilon^v(x)F_1(M)+(1-h_\varepsilon^v(x))F_2(M).
\end{equation}
This auxiliary operator leads to the Dirichlet problem
\begin{equation}\label{eq_nm1/2}
	\begin{cases}
		F_{\varepsilon}^v(D^2u,u,x)=f&\hspace{.2in}\mbox{in}\hspace{.1in}\Omega\\
		u=g&\hspace{.2in}\mbox{in}\hspace{.1in}\partial\Omega.
	\end{cases}
\end{equation}
In \cite{PimSwi}, the authors establish a comparison principle for \eqref{eq_nm0}, while ensuring the existence of global barriers for \eqref{eq_nm1/2}. An application of Perron's method ensures the existence of a (unique) viscosity solution $u_\varepsilon^v\in C(\overline\Omega)$ to \eqref{eq_nm1/2}. 

To recover a viscosity solution to \eqref{eq_main1}, the authors examine the two-parameter family $(u_\varepsilon^v)_{\varepsilon,v}$. The analysis of the functional parameter $v$ depends on a set $\mathcal{B}\subset C(\overline\Omega)$ and a map $T:\mathcal{B}\to C(\overline\Omega)$. For $0<\varepsilon\ll1$ fixed, this operator is given by $Tv\coloneqq u_\varepsilon^v$. The authors prove that $T$ has a fixed point, leading to the existence of a solution to
\[
	\begin{cases}
		F_{\varepsilon}^{u_\varepsilon}(D^2u_\varepsilon,u_\varepsilon,x)=f&\hspace{.2in}\mbox{in}\hspace{.1in}\Omega\\
		u_\varepsilon=g&\hspace{.2in}\mbox{in}\hspace{.1in}\partial\Omega.
	\end{cases}
\]
Finally, they take the limit $\varepsilon\to 0$. Regularity estimates for the family $(u_\varepsilon)_{0<\varepsilon\ll1}$ ensure the existence of $u\in C(\overline\Omega)$ such that $u_\varepsilon\to u$, as $\epsilon\to 0$, through a subsequence, if necessary. The stability of viscosity solutions implies that $u$ solves \eqref{eq_main1} in the viscosity sense.

We propose a numerical method for \eqref{eq_main1} by discretising \eqref{eq_nm1/2}. Indeed, for $0<h\ll1$ we consider a dicrete approximation of $\overline\Omega$, denoted with $\overline\Omega_h$ and examine the method
\begin{equation}\label{eq_nm1}
	G_{\varepsilon,h}^v(u_{\varepsilon,h}^v(x),x)\coloneqq
		\begin{cases}
			F_\varepsilon^v(D_h^2u_{\varepsilon,h}^v(x),u_{\varepsilon,h}^v(x),x)-f(x)&\hspace{.2in}\mbox{in}\hspace{.1in}\Omega_h\\
			u_{\varepsilon,h}^v(x)-g(x)&\hspace{.2in}\mbox{on}\hspace{.1in}\partial\Omega_h,
		\end{cases}
\end{equation}
where $D^2_h$ is a discrete approximation of the Hessian matrix, and $u_{\varepsilon,h}^v$ is a grid function. We start by establishing the existence of a unique solution to
\begin{equation}\label{eq_nm2}
	G_{\varepsilon,h}^v(u_{\varepsilon,h}^v(x),x)=0\hspace{.2in}\mbox{in}\hspace{.1in}\overline\Omega_h.
\end{equation}
Then, for $N\in\mathbb{N}$ depending only on $h$, we design a subset $\mathcal{B}_h\subset \mathbb{R}^N$ and a map $T:\mathcal{B}_h\to \mathbb{R}^N$. Properties of \eqref{eq_nm2} allow us to apply Schauder's Fixed Point Theorem to ensure the existence of $u_{\varepsilon,h}$ solving
\begin{equation}\label{eq_nm3}
	G_{\varepsilon,h}^{u_{\varepsilon,h}}(u_{\varepsilon,h}(x),x)=0\hspace{.2in}\mbox{in}\hspace{.1in}\overline\Omega_h.
\end{equation}
Once the existence of a solution to \eqref{eq_nm3} is available, we examine the numerical method described by $G_{\varepsilon,h}^{u_{\varepsilon,h}}$. We use a centred differences discretisation of the Hessian to ensure the monotonicity of the method. Also, we combine global barriers with a discrete version of the comparison principle to establish stability of $G_{\varepsilon,h}^{u_{\varepsilon,h}}$. Finally, the regularity of the operator $F_{\varepsilon}^{v}$ yields consistency of the method with \eqref{eq_nm1/2}.

Monotonicity, stability and consistency build upon convergence results in the literature to ensure that $u_{\varepsilon,h}\to u_\varepsilon$ locally uniformly, as $h\to 0$, where $u_\varepsilon\in C(\overline\Omega)$ solves \eqref{eq_nm1/2} in the viscosity sense; see \cite[Theorem 2]{BarSou}. As in the continuous case, we rely on regularity estimates available for the solutions to \eqref{eq_nm1/2}. As a result, the family of numerical solutions $(u_\varepsilon)_{0<\varepsilon\ll1}$ converges locally uniformly to a viscosity solution to \eqref{eq_main1}. Our main result is the following theorem.

\begin{theorem}[Convergence of the numerical scheme]\label{teo_main}
Suppose Assumptions A\ref{assump_Omega}-A\ref{assump_data}, to be detailed further, are in force. Fix $0<\varepsilon\ll1$. There exists $0<h_0\ll1$ such that, if $h\in(0,h_0)$, then \eqref{eq_nm3} has a unique solution $u_{\varepsilon,h}$. In addition, as $h\to 0$, $u_{\varepsilon,h}\to u_\varepsilon$ locally uniformly in $\Omega$, where $u_\varepsilon$ is the unique viscosity solution to \eqref{eq_nm1/2}. Finally, we have $u_\varepsilon\to u$, as $\varepsilon\to 0$, where $u$ solves \eqref{eq_main1} in the viscosity sense.
\end{theorem}

\begin{remark}[Main assumptions]\label{rem_assump}\normalfont
Among our main assumptions, we require the domain $\Omega$ to satisfy a uniform exterior cone condition. This is critical for the existence of viscosity solutions to \eqref{eq_main1}; see \cite{PimSwi}. We further assume $F_1$ and $F_2$ are operators of the Isaacs type. This is a natural assumption in the study of fully nonlinear elliptic equations; see \cite[Remark 1.5]{CabCaf}. For simplicity, we suppose the sets of matrices governing $F_1$ and $F_2$ satisfy a diagonal dominance condition. However, this assumption can be completely relaxed; see, for instance, \cite{DebJak}.
\end{remark}

\begin{remark}[Non-uniqueness and a selection criteria]\label{rem_selection}\normalfont
We recall that the uniqueness of solutions to \eqref{eq_main1} remains as an open problem. Therefore, the numerical approximation designed by Theorem \ref{teo_main} can be used to select families of solutions to \eqref{eq_main1}, provided they retain improved qualitative properties.
\end{remark}

\begin{remark}[Fixed-point analysis]\label{rem_fpa}\normalfont
One of the main contributions of the present paper is the analysis of a fixed-point argument leading to the existence of a (unique) solution to \eqref{eq_nm3}. Our strategy appears robust and can be adapted to a wider latitude of problems. We first mention the numerical treatment of mean-field game systems. In addition, it can simplify the numerical analysis of poliharmonic problems, once they are written in the form of a system of Poisson equations \cite{AlcPimUrb}. Finally, we believe it could be useful in the computational treatment of $p$-Poisson equations, provided one can write $\Delta_pu=f$ as
\[
	\begin{cases}
		{\rm div}\left(m Du\right)=f&\hspace{.2in}\mbox{in}\hspace{.2in}\Omega\\
		|Du|^2=m^\frac{2}{p-2}&\hspace{.2in}\mbox{in}\hspace{.2in}\Omega.
	\end{cases}
\]  
\end{remark}

The remainder of this manuscript is organised as follows. Section \ref{sec_prelim} gathers preliminary material used throughout the paper. We detail our numerical method and present the proof of Theorem \ref{teo_main} in Section \ref{sec_method}. Numerical examples validating our method are the subject of Section \ref{sec_examples}.

\section{Preliminaries}\label{sec_prelim}

We start with the main assumptions used in the paper. First, we impose a uniform exterior cone condition on the domain $\Omega$.

\begin{assumption}[Domain's geometry]\label{assump_Omega}
We suppose $\Omega\subset\mathbb{R}^d$ satisfies a uniform exterior cone condition. That is, there exists $r,\theta>0$ such that, for every $x\in\partial\Omega$, one can find a cone $C_\theta$ of opening $\theta$ centred at the origin, such that 
\[
	\left((x+C_\theta)\cap B_r(x)\right)\subset \mathbb{R}^d\setminus \Omega.
\]
\end{assumption}

Next, we detail our assumption on the uniform ellipticity of the operators $F_1$ and $F_2$.

\begin{assumption}[Uniform ellipticity]\label{assump_Felliptic}
Fix $0<\lambda\leq\Lambda$. We suppose the operators $F_1$ and $F_2$ are $(\lambda,\Lambda)$-uniformly elliptic. That is, for every $M,N\in S(d)$, we have
\[
	\lambda\left\|N\right\|\leq F(M)-F(M+N)\leq \Lambda\left\|N\right\|,
\]
provided $N\geq 0$.
\end{assumption}

We suppose that $F_1$ and $F_2$ are operators of Isaacs type, driven by symmetric matrices satisfying a diagonal dominance condition.

\begin{assumption}[Diagonal dominance]\label{assump_diagonal}
For $i=1,2$, we suppose
\[
	F_i(M)\coloneqq\sup_{\alpha\in\mathcal{A}}\inf_{\beta\in\mathcal{B}}{\rm Tr}\left(A_{\alpha,\beta}^iM\right),
\]
where $A_{\alpha,\beta}^i\coloneqq\left(a_{j,k}^{\alpha,\beta,i}\right)_{j,k=1}^d$ is a negative semi-definite matrix satisfying
\[
	\left|a_{j,j}^{\alpha,\beta,i}\right|\geq \sum_{\substack{j,k=1\\j\neq k}}^d\left|a_{j,k}^{\alpha,\beta,i}\right|,
\]
for every $j=1,\ldots,d$, $\alpha\in\mathcal{A}$, and $\beta\in\mathcal{B}$, for every $i=1,2$.
\end{assumption}

\begin{assumption}[Data of the problem]\label{assump_data}
We suppose $f\in L^\infty(\Omega)\cap C(\Omega)$ and $g\in C(\partial\Omega)$. 
\end{assumption}

Once our assumptions have been stated, we recall preliminary facts about the discretisation used in the paper. Define $\widetilde\Omega_h$ as 
\[
	\widetilde\Omega\coloneqq\left\lbrace x\in\mathbb{R}^d\,|\,{\rm dist}(x,\Omega)>h\right\rbrace. 
\]
For $0<h\ll1$, consider a uniform discretisation of $\widetilde\Omega$ of grid size $h$, denoted by $\widetilde\Omega_h$. Consider the Dirichlet problem
\begin{equation}\label{eq_pde1}
	\begin{cases}
		F(D^2u,Du,u,x)=f&\hspace{.2in}\mbox{in}\hspace{.1in}\Omega\\
		u=g&\hspace{.2in}\mbox{on}\hspace{.1in}\partial\Omega,
	\end{cases}
\end{equation}
where $F:S(d)\times\mathbb{R}^d\times\mathbb{R}\times\Omega\to\mathbb{R}$ is a $(\lambda,\Lambda)$-elliptic operator. We propose a numerical approximation of \eqref{eq_pde1} in $\widetilde\Omega_h$ of the form
\begin{equation}\label{eq_met1}
	\begin{cases}
		F_h(D^2_hu_h,D_hu_h,u_h,x)=f&\hspace{.2in}\mbox{in}\hspace{.1in}\widetilde\Omega_h\cap\Omega\\
		u=g&\hspace{.2in}\mbox{on}\hspace{.1in}\widetilde\Omega_h\setminus\Omega.
	\end{cases}
\end{equation}
In \eqref{eq_met1}, $F_h(\cdot,\cdot,\cdot,x)$ accounts for the restriction of $G$ to $\widetilde\Omega_h\cap\Omega$, whereas $D_h^2v$ and $D_hv$ stand, respectively, for a discrete approximation of the Hessian and the gradient of the function $v$. The unknown $u_h:\widetilde\Omega_h\to\mathbb{R}^d$ is a grid function satisfying \eqref{eq_met1} in $\widetilde\Omega_h$. The discretization \eqref{eq_met1} is written as
\begin{equation}\label{eq_met2}
	G_h(u_h,x)\coloneqq
		\begin{cases}
			F_h(D^2_hu_h,D_hu_h,u_h,x)-f&\hspace{.2in}\mbox{in}\hspace{.1in}\widetilde\Omega_h\cap\Omega\\
			u-g&\hspace{.2in}\mbox{on}\hspace{.1in}\widetilde\Omega_h\setminus\Omega.
		\end{cases}
\end{equation}

\begin{remark}[Extended domain]\label{rem_extdom}\normalfont
We prescribe \eqref{eq_met1} and \eqref{eq_met2} in the extended domain $\widetilde\Omega_h$ to account for the geometry of $\partial\Omega$. Since we work with a uniform grid of size $h$, it might be the case that in a vicinity of $\partial\Omega$ the grid size skips boundary points. Therefore, by prescribing $u=g$ in $\widetilde \Omega_h\setminus\Omega$, one ensures the boundary data is verified as $h\to 0$. Although we prescribe \eqref{eq_met1} and \eqref{eq_met2} in the extended domain, in the remainder of the paper, we work in $\overline\Omega_h$ for ease of presentation.
\end{remark}

For completeness, we include the definitions of monotone, stable and consistent numerical schemes. 

\begin{definition}[Monotonicity]\label{def_monotonicity}
The numerical method $(G_h)_{0<h\ll1}$ is monotone if, for every $0<h\ll1$, and every $u_h,v_h:\overline\Omega_h\to\mathbb{R}$ satisfying $u_h(x)=v_h(x)$, for some $x\in \overline\Omega_h$, with $u_h\leq v_h$, we have
\[
	G_h(v_h(x),x)\leq G_h(u_h(x),x).
\]
\end{definition}

\begin{definition}[Stability]\label{def_stability}
The numerical method $(G_h)_{0<h\ll 1}$ is stable if there exists $C>0$ such that 
\[
	\sup_{0<h\ll1}\max_{x\in\overline\Omega_h}\left|u_h(x)\right|\leq C.
\]
The constant $C>0$ is allowed to depend on the data of the problem, but not on the grid size $0<h\ll1$.
\end{definition}
To define consistency, it is useful to recall the notion of upper and lower envelopes associated with a function $G:S(d)\times\mathbb{R}^d\times\mathbb{R}\times\overline\Omega\to\mathbb{R}$. Set 
\begin{equation}\label{eq_pde11}
	G(M,p,r,x)\coloneqq
		\begin{cases}
			F(M,p,r,x)-f(x)&\hspace{.2in}\mbox{if}\hspace{.1in}x\in\Omega\\
			r-g(x)&\hspace{.2in}\mbox{if}\hspace{.1in}x\in\partial\Omega.
		\end{cases}
\end{equation}

\begin{definition}[Upper and lower envelopes]\label{def_ule}
We define the upper envelope $G^*$ of \eqref{eq_pde11} as
\[
	G^*(M,p,r,x_0)\coloneqq \limsup_{x\to x_0}G(M,p,r,x).	
\]
The lower envelope $G_*$ associated with \eqref{eq_pde11} is
\[
	G_*(M,p,r,x_0)\coloneqq \liminf_{x\to x_0}G(M,p,r,x).	
\]
\end{definition}

Now, we define the consistency of a numerical method in terms of the envelopes associated with the Dirichlet problem it approximates.

\begin{definition}[Consistency]\label{def_consistency}
The numerical method $(G_h)_{0<h<h_0}$ is consistent with \eqref{eq_pde1} if
\[
	\limsup_{\substack{h\to 0\\y\to x\\\xi\to 0}}G_h(\varphi(y)+\xi,y)\leq G^*(D^2\varphi(x), D\varphi(x), \varphi(x),x)
\] 
and
\[
	\liminf_{\substack{h\to 0\\y\to x\\\xi\to 0}}G_h(\varphi(y)+\xi,y)\geq G_*(D^2\varphi(x), D\varphi(x), \varphi(x),x),
\] 
whenever $\varphi\in C^\infty(\overline\Omega)$ and $x\in \overline\Omega$.
\end{definition}

In the sequel, we recall a characterisation of convergence for numerical methods approximating viscosity solutions. It relies on the monotonicity, stability and consistency of the numerical method; see the seminal work of Barles and Souganidis \cite{BarSou}.

\begin{proposition}[Convergence of the numerical method]\label{prop_bs91}
Let $(G_h)_{0<h\ll1}$ be a monotone and stable numerical method. Let $(u_h)_{0<h\ll1}$ be a family of grid functions such that $u_h:\overline\Omega_h\to\mathbb{R}$ solves
\[
	G_h(u_h(x),x)=0
\]
for every $x\in\overline\Omega_h$. $(G_h)_{0<h\ll1}$ is consistent with \eqref{eq_pde1}, then $u_h\to u$ locally uniformly in $\Omega$, where $u\in C(\Omega)$ is a viscosity solution to \eqref{eq_pde1}. 
\end{proposition}

For the proof of Proposition \ref{prop_bs91}, we refer the reader to \cite[Theorem 2]{BarSou}. We conclude this section with a discussion on the solvability of $G_h(u_h,x)=0$, as outlined in \cite{Obe}. Consider an Euler operator of the form
\[
	S(\rho,h,u_h,x)\coloneqq u_h-\rho G_h(u_h,x),
\]
where $0<\rho\ll1$ is a small parameter related to $h$ through a (CFL) condition. The latter depends on the ellipticity of $F$, the dimension $d$ and the data of the problem. Under a CFL condition, one proves that $S$ is a contraction, which ensures the existence of a unique fixed point in an appropriate Banach space. This is tantamount to the existence of a (unique) solution to $G_h=0$. This strategy is implemented in our numerical examples; see Section \ref{sec_examples}.

\section{A finite difference method}\label{sec_method}

We propose a numerical method based on the strategy implemented to prove the existence of solutions to \eqref{eq_main1}; see \cite[Theorems 1 and 2]{PimSwi}. For $0<h\ll1$, denote with $\overline\Omega_h$ a discrete approximation of $\overline\Omega$. Let $N(h)$ stand for the cardinality of $\overline\Omega_h$. When useful, we identify $\overline\Omega_h$ with a point in $\mathbb{R}^{d\times N(h)}$. We also consider an enumeration for the points in $\overline\Omega_h$, given by
\[
	\left(x_1,x_2,\ldots,x_{N(h)}\right)\in\mathbb{R}^{d\times N(h)}.
\]

Fix $0<\varepsilon\ll1$ and let $v:\overline\Omega_h\to\mathbb{R}$. Notice that $v(\overline\Omega_h)\in\mathbb{R}^{N(h)}$; for simplicity, we sometimes write
\[
	v=\left(v_1,\ldots,v_{N(h)}\right)=\left(v(x_1),\ldots,v(x_{N(h)})\right),
\]
where we have used the enumeration of the points in $\overline\Omega_h$.
We define $h_\varepsilon^v:\Omega_h\to\mathbb{R}$ as
\[
	h_\varepsilon^v(x)\coloneqq\max\left(\min\left(\frac{v(x)+\varepsilon}{2\varepsilon},1\right),0\right).
\]
Consider the operator $F_{\varepsilon,h}^v:\mathbb{R}^{d\times d}\times\mathbb{R}\times\Omega_h\to\mathbb{R}$ given by
\[
	F^v_{\varepsilon,h}(M,r,x)\coloneqq\varepsilon r+h_\varepsilon^v(x)F_1(M)+(1-h_\varepsilon^v(x))F_2(M)-f(x).
\]
We denote the discretisation of the Hessian with
\[
	D^2_h w(x)\coloneqq\left(\partial^2_{x_i,x_j}w(x)\right)_{i,j=1}^d,
\]
where
\[
	\partial^2_{x_i,x_i}w_h(x)\coloneqq\frac{w_h(x+he_i)+w_h(x-he_i)-2w_h(x)}{h^2}
\]
and
\[
	\begin{split}
		\partial^2_{x_i,x_j}w_h(x)&\coloneqq \frac{-2w_h(x) + u_h(x + e_i h - e_j h) + w_h(x - e_i h + e_j h)}{2h^2}\\
			&\quad+ \frac{w_h(x + e_i h) + w_h(x - e_i h) + w_h(x + e_j h) + w_h(x - e_j h)}{2h^2}.
	\end{split}
\]
Working under the Assumption A\ref{assump_diagonal}, we obtain
\[
	F_i(D^2_hu_h(x))\leq F_i(D^2_hw_h(x))
\]
whenever $w_h\leq u_h$. We study the numerical method given by
\begin{equation}\label{eq_method}
	G_{\varepsilon,h}^v(u_h,x)\coloneqq
		\begin{cases}
			F_{\varepsilon,h}^v(D^2u_h,u_h,x)-f(x)&\hspace{.1in}\mbox{if}\hspace{.1in}x\in\Omega_h\\
			u_h(x)-g(x)&\hspace{.1in}\mbox{if}\hspace{.1in}x\in\partial\Omega_h,
		\end{cases}
\end{equation}
where the unknown $u_h:\overline\Omega_h\to\mathbb{R}$ is a function whose image $u_h(\overline\Omega_h)$ has cardinality $N(h)$. As before, $u_h(\overline\Omega_h)\in \mathbb{R}^{N(h)}$ and we write
\[
	\left(u_h^1,u_h^2,\ldots,u_h^{N(h)}\right)\coloneqq\left(u_h(x_1),u_h(x_2),\ldots,u_h(x_{N(h)})\right),
\]
using once again the enumeration of the points in $\overline\Omega_h$. We proceed by verifying that $G_{\varepsilon,h}^v$ is monotone for every $0<\varepsilon,h\ll1$ and every $v\in \mathbb{R}^{N(h)}$.

\begin{proposition}[Monotonicity]\label{prop_monotone}
Suppose Assumptions A\ref{assump_diagonal} and A\ref{assump_data} hold. Then $G_{\varepsilon,h}^v$ is monotone for every $0<\varepsilon,h\ll1$ and every $v\in \mathbb{R}^{N(h)}$.
\end{proposition}
\begin{proof}
Let $u_h,w_h:\overline\Omega_h\to\mathbb{R}$ be such that $u_h(x)=w_h(x)$ for some $x\in\overline\Omega_h$, with $u_h\leq w_h$ in $\overline\Omega_h$. We must verify that $G_{\varepsilon,h}^v(u_h,x)\geq G_{\varepsilon,h}^v(w_h,x)$.

If $x\in \partial\Omega_h$, we have
\[
	G_{\varepsilon,h}^v(u_h,x)=u_h(x)-g(x)=w_h(x)-g(x)=G_{\varepsilon,h}^v(w_h,x),
\]
and the inequality follows.

Otherwise, suppose $x\in \Omega_h$. Then
\[
	\begin{split}
		G_{\varepsilon,h}^v(u_h,x)=F_{\varepsilon,h}^v(D^2u_h,u_h,x)\geq F_{\varepsilon,h}^v(D^2w_h,w_h,x)=G_{\varepsilon,h}^v(w_h,x),
	\end{split}
\]
and the result follows.
\end{proof}

We continue with the stability analysis for $G_{\varepsilon,h}^v$. To that end, we introduce barrier functions $\underline w$ and $\overline w$ and consider their restrictions to $\overline\Omega_h$. 

\begin{proposition}[Discrete global barriers]\label{prop_dgb}
Suppose Assumptions A\ref{assump_diagonal} and A\ref{assump_data} hold true. There exist $\underline w,\overline w:\overline\Omega\to\mathbb{R}$ such that 
\[
	G_{\varepsilon,h}^v(\underline w(x),x)\leq 0\leq G_{\varepsilon,h}^v(\overline w(x),x)
\]
for every $0<\varepsilon\ll1$, $v\in\mathbb{R}^{N(h)}$, and  $x\in\overline\Omega_h$, for all $0<h\ll 1$.
\end{proposition}
\begin{proof}
We detail the case of $\overline w$, as the remaining one is completely analogous. For $x\in\overline\Omega_h$, set
\[
	\overline w(x)\coloneqq C_1-\frac{C_2}{2\lambda d}\left|x\right|^2\geq C_1-\frac{L^2}{2\lambda d},
\]
where $L>0$ is such that $\Omega\subset B_L$. Choose $C_1>0$ large enough such that $\overline w\geq \left\|g\right\|_{L^\infty(\partial\Omega)}$. For $x\in\partial\Omega_h$, such a choice ensures 
\[
	G_{\varepsilon,h}^v(\overline w(x),x)=\overline w(x)-g(x)\geq \left\|g\right\|_{L^\infty(\partial\Omega)}-g(x)\geq 0.
\]
For $x\in \Omega_h$, we get
\[
	\begin{split}
		G_{\varepsilon,h}^v(\overline w(x),x)&=\varepsilon\left(C_1-\frac{C_2}{2\lambda d}\left|x\right|^2\right)+h_\varepsilon^v(x)F_1\left(\frac{-C_2}{\lambda d}Id\right)\\
			&\quad+(1-h_\varepsilon^v(x))F_2\left(\frac{-C_2}{\lambda d}Id\right)-f(x)\\
			&\geq \left\|g\right\|_{L^\infty(\partial\Omega)}+C_2-f(x).
	\end{split}
\]
By choosing $C_2\geq \left\|f\right\|_{L^\infty(\Omega)}$ one completes the argument. By setting $\underline w\coloneqq \overline w$, one finishes the argument.
\end{proof}

The stability is the subject of the next proposition. Here, the discrete global barriers play a central role.

\begin{proposition}[Stability]\label{prop_stability}
Suppose Assumptions A\ref{assump_diagonal} and A\ref{assump_data} hold true. Let $0<h\ll1$ and $u_h:\overline\Omega_h\to\mathbb{R}$ be a solution to $G_{\varepsilon,h}^v(u_h,x)=0$. There exists a constant $C>0$ such that 
\[
	\sup_{0<h\ll1}\max_{x\in\overline\Omega_h}\left|u_h(x)\right|\leq C.
\]
Moreover, $C>0$ depends only on the dimension $d$, the ellipticity constants $\lambda$ and $\Lambda$, $\left\|f\right\|_{L^\infty(\Omega)}$ and $\left\|g\right\|_{L^\infty(\partial\Omega)}$.
\end{proposition}
\begin{proof}
To prove the proposition, it suffices to verify that 
\[
	\underline w(x)\leq u_h(x)\leq \overline w(x)
\]
for every $x\in \overline\Omega_h$, for every $h\in(0,1)$. We verify the second inequality, as the first one follows from an entirely analogous reasoning. 

We claim that $u_h\leq \overline w$ in $\overline\Omega_h$. Suppose otherwise; there exists $x\in\overline\Omega_h$ such that $\overline w(x)<u_h(x)$. If $x\in\partial\Omega_h$, we have
\[
	0=G_{\varepsilon,h}^v(u_h(x),x)=u_h(x)-g(x)>\overline w(x)-g(x)=G_{\varepsilon,h}^v(\overline w(x),x),
\]
which is a contradiction to $G_{\varepsilon,h}^v(\overline w(x),x)\geq 0$. Suppose $x\in \Omega_h$. In this case, suppose without loss of generality that $x$ is a maximum point for $u_h-\overline w$. That is,
\[
	u_h(x)-\overline w(x)\geq u_h(y)-\overline w(y)
\]
for every $y\in\Omega_h$. Then
\begin{equation}\label{eq_stability1}
	u_h(x)-u_h(y)> \overline w(x)-\overline w(y).
\end{equation}
Hence,
\[
	\begin{split}
		G_{\varepsilon,h}^v(u_h(x),x)&=F_{\varepsilon,h}^v(D^2_hu_h(x),u_h(x),x)\\
			&>F_{\varepsilon,h}^v(D_h^2\overline w(x),\overline w(x),x)\\
			&=G_{\varepsilon,h}^v(\overline w(x),x),
	\end{split}	
\]
where the strict inequality follows from \eqref{eq_stability1}. Because $u_h$ is a solution to $G_{\varepsilon,h}^v=0$, the former inequality yields a contradiction and proves the proposition.
\end{proof}

We proceed with a fixed-point analysis. Our goal is to show there exists $u_h:\overline\Omega_h\to\mathbb{R}$ such that 
\[
	G_{\varepsilon,h}^{u_h}(u_h(x),x)=0
\]
for every $x\in \overline\Omega_h$.

For the barriers $\underline w,\overline w:\overline \Omega\to\mathbb{R}$ and \emph{fixed} $0<h<1$, denote with $\mathcal{B}_h$ the set
\begin{equation}\label{eq_bh}
	\mathcal{B}_h\coloneqq\left\lbrace w\in\mathbb{R}^{N(h)}\,|\,\min_{i=1,\ldots,N(h)}\underline w(x_i)\leq w_n\leq  \max_{i=1,\ldots,N(h)}\overline w(x_i)\right\rbrace.
\end{equation}
Define also $T_h:\mathcal{B}_h\to\mathbb{R}^{N(h)}$ as follows. For $v\in \mathcal{B}_h$, $Tv$ is the unique solution to $G_{\varepsilon,h}^v=0$ in $\overline\Omega_h$. It is clear that $\mathcal{B}_h$ is closed and convex. We continue with the properties of the map $T$.

\begin{proposition}[Properties of the map $T$]\label{prop_T}
Let $0<h<h_0$ be fixed and define $\mathcal{B}_h$ as in \eqref{eq_bh}. Then $T(\mathcal{B}_h)\subset \mathcal{B}_h$. Additionally, $T$ is continuous and precompact.
\end{proposition}
\begin{proof}
For the sake of clarity, we split the proof into three steps. First, we address the invariance of $\mathcal{B}_h$ under the map $T$.

\smallskip

\noindent{\bf Step 1 - }Let $v\in\mathcal{B}_h$. Then $Tv$ is the unique solution to $G_{\varepsilon,h}^v=0$. Proposition \ref{prop_stability} ensures that $\underline w(x)\leq Tv(x)\leq \overline w(x)$, for every $x\in \overline\Omega_h$. Hence, 
\[
	\min_{j=1,\ldots,N(h)}\underline w(x_j)\leq Tv(x_i)\leq  \max_{j=1,\ldots,N(h)}\overline w(x_j)
\]
for every $i=1,\ldots,N(h)$. We conclude that $Tv\in \mathcal{B}_h$.

\smallskip

\noindent{\bf Step 2 - }To verify that $T$ is continuous, we let $(v_n)_{n\in\mathbb{N}}$ and suppose there exists $v_\infty\in \mathcal{B}_h$ such that $v_n\to v_\infty$, as $n\to\infty$. We claim that $Tv_n\to Tv_\infty$, as $n\to\infty$.

Indeed, denote with $u_{\varepsilon,h}^{n}$ the unique solution to $G_{\varepsilon,h}^{v_n}$. The sequence $(u_n)_{n\in\mathbb{N}}$ is uniformly bounded, because of Proposition \ref{prop_stability}. Compactness in $\mathbb{R}^{N(h)}$ ensures the existence of a subsequence $(u_{n_k})_{k\in\mathbb{N}}$, converging to a function $\overline u\in\mathcal{B}_h$. Now, take $x\in \overline \Omega_h$. If $x\in \partial\Omega_h$, we have
\begin{equation}\label{eq_conT1}
	0=G_{\varepsilon,h}^{v_{n_k}}(u_{n_k}(x),x)=u_{n_k}(x)-g(x)\to\overline u(x)-g(x).
\end{equation}
If $x\in\Omega_h$, we have 
\begin{equation}\label{eq_conT2}
	0=G_{\varepsilon,h}^{v_{n_k}}(u_{n_k}(x),x)\to h_\varepsilon^{v_\infty}(x)F_1(D_h^2\overline u(x))+(1-h_\varepsilon^{v_\infty}(x))F_2(D_h^2\overline u(x))
\end{equation}
Combining \eqref{eq_conT1} and \eqref{eq_conT2}, one obtains
\[
	G_{\varepsilon,h}^{v_\infty}(\overline u(x),x)=0
\]
for every $x\in \overline\Omega_h$. The uniqueness of solutions to $G_{\varepsilon,h}^{v_\infty}=0$ implies that 
\[
	Tv_\infty=\overline u=\lim_{n\to\infty}Tv_n,
\]
independently of the subsequence $(u_{n_k})_{k\in\mathbb{N}}$, ensuring the continuity of $T$. 

\smallskip

\noindent{\bf Step 3 - }It remains to verify that $T$ is precompact. Let $(Tv_n)_{n\in\mathbb{N}}\subset T(\mathcal{B}_h)$. Once again, stability (Proposition \ref{prop_stability}) ensures that $(Tv_n)_{n\in\mathbb{N}}$ is a uniformly bounded subset of $\mathbb{R}^{N(h)}$. Clearly, it admits a convergent subsequence, and the proof is complete.
\end{proof}

A corollary of Proposition \ref{prop_T} is the existence of $u_h\in \mathcal{B}_h$ satisfying
\begin{equation}\label{eq_Gpf}
	G_{\varepsilon,h}^{u_h}(u_h(x),x)=0
\end{equation}
fo every $x\in \overline\Omega_h$.

\begin{corollary}\label{cor_pf}
Let $0<h<h_0$ be fixed and define $\mathcal{B}_h$ as in \eqref{eq_bh}. There exists $u_h\in \mathcal{B}_h$ satisfying \eqref{eq_Gpf} for every $x\in \overline\Omega_h$.
\end{corollary}
\begin{proof}
The corollary follows from Proposition \ref{prop_T} combined with Schauder's Fixed Point Theorem; see, for instance \cite[Theorem 3, Section 9.2.2]{Eva1998}.
\end{proof}

The previous corollary ensures that, for every $0<\varepsilon,h\ll1$, there exists $u_h:\overline\Omega_h\to\mathbb{R}$ such that \eqref{eq_Gpf} is satisfied. Our first goal is to take the limit $h\to 0$. Since we have already verified that $G_{\varepsilon,h}^{u_h}$ is monotone and stable, it remains to prove it is consistent with \eqref{eq_nm1/2}.

\begin{proposition}[Consistency of $G_{\varepsilon,h}^{u_h}$]\label{prop_consistency}
Let $G_{\varepsilon,h}^v$ be defined as in \eqref{eq_method}. Suppose assumptions A\ref{assump_diagonal} and A\ref{assump_data} are in force. For $u\in C(\overline\Omega)$, the method $G_{\varepsilon,h}^{u}$ is consistent with 
\begin{equation}\label{eq_nm1/3}
	\begin{cases}
		F_{\varepsilon}^u(D^2u,u,x)=f&\hspace{.2in}\mbox{in}\hspace{.1in}\Omega\\
		u=g&\hspace{.2in}\mbox{in}\hspace{.1in}\partial\Omega.
	\end{cases}
\end{equation}
\end{proposition}
\begin{proof}
Let $\varphi\in C^\infty(\overline\Omega)$ and fix $x\in\overline\Omega$. We aim to prove that 
\begin{equation}\label{eq_cons1a}
	\limsup_{\substack{h\to 0\\y\to x\\\xi\to 0}}G_{\varepsilon,h}^{\varphi+\xi}(\varphi(y)+\xi,y)\leq G^*\left(D^2\varphi(x),\varphi(x),x\right).
\end{equation}
For ease of presentation, we split the proof into two cases.

\smallskip

\noindent{\bf Case 1 - }Suppose $x\in\Omega$. In that case, the points $y$ approaching $x$ are also in the interior $\Omega$. Hence,
\begin{align*}
		G_{\varepsilon,h}^{\varphi+\xi}(\varphi(y)+\xi,y)&=\varepsilon (\varphi(y)+\xi)+h_\varepsilon^{\varphi+\xi}(y)F_1(D^2_h(\varphi(y)+\xi))\\
			&\quad+(1-h_\varepsilon^{\varphi+\xi}(y))F_2(D^2_h(\varphi(y)+\xi)).
\end{align*}

We also know that 
\[
	D_h^2(\varphi(y)+\xi)= D^2\varphi(x)+O(h^2)\hspace{.2in}\mbox{and}\hspace{.2in}D_h(\varphi(y)+\xi)= D\varphi(x)+O(h).
\]
As a consequence, the continuity of $F_1$, $F_2$, and $h_\varepsilon^{\varphi+\xi}$ ensure that 
\begin{align*}
		\limsup_{\substack{h\to 0\\y\to x\\\xi\to 0}}G_{\varepsilon,h}^{\varphi+\xi}(\varphi(y)+\xi,y)&=\varepsilon\varphi(x)+h_\varepsilon^{\varphi}(x)F_1(D^2\varphi(x))\\
			&\quad+(1-h_\varepsilon^{\varphi}(x))F_2(D^2\varphi(x))\\
			&\leq G^*(D^2\varphi(x),\varphi(x),x).
\end{align*}
We proceed with the case $x\in\partial\Omega$.

\smallskip

\noindent{\bf Case 2 - }Suppose now that $x\in\partial\Omega$. In this case, we can consider that either the approaching points satisfy $y\in\Omega_h$ or $y\in\partial\Omega_h$; this is due to the limit superior operation. 

Let $y\in\partial\Omega$. Then
\[
	G_{\varepsilon,h}^{\varphi+\xi}(\varphi(y)+\xi,y)=\varphi(y)+\xi-g(y)\longrightarrow\varphi(x)-g(x)\leq G^*(D^2\varphi(x),\varphi(x),x),
\]
as $h\to0$, $y\to x$ and $\xi\to 0$. Hence, the desired inequality follows. Now, let $y\in\Omega_h$. As before, we obtain
\begin{align*}
		&\limsup_{\substack{h\to 0\\y\to x\\\xi\to 0}}G_h(\varphi(y)+\xi,y)\leq G^*\left(D^2\varphi(x),D\varphi(x),\varphi(x),x\right).
\end{align*}
In any case, \eqref{eq_cons1a} is verified, and the method is consistent with \eqref{eq_nm1/3}.
\end{proof}

In the sequel, we present the proof of Theorem \ref{teo_main}.

\begin{proof}[Proof of Theorem \ref{teo_main}]
Let $(u_{\varepsilon,h})_{0<h\ll1}$ be the family of solutions to
\[
	G_{\varepsilon,h}^{u_{\varepsilon,h}}(u_{\varepsilon,h}(x),x)=0
\]
in $\overline\Omega_h$. Combine Corollary \ref{cor_pf} and Proposition \ref{prop_consistency} with Proposition \ref{prop_bs91} to conclude that $u_{\varepsilon,h}\to u_\varepsilon$, where $u_\varepsilon$ is a viscosity solution to 
\[
	\begin{cases}
		\varepsilon u_\varepsilon+h_\varepsilon^{u_\varepsilon}F_1(D^2u_\varepsilon)+(1-h_\varepsilon^{u_\varepsilon})F_2(D^2u_\varepsilon)=f&\hspace{.1in}\mbox{in}\hspace{.1in}\Omega\\
		u_\varepsilon=g&\hspace{.1in}\mbox{on}\hspace{.1in}\partial\Omega.
	\end{cases}
\]
Arguing as in \cite[Theorem 1]{PimSwi}, one concludes that $u_\varepsilon\to u$ locally uniformly in $\Omega$, where $u$ is a viscosity solution to \eqref{eq_main1}.
\end{proof}

\section{Numerical examples}\label{sec_examples}

In this section, we present two illustrative examples demonstrating the convergence of the numerical method 
in both one-dimensional and two-dimensional settings.

The approximate solution of problem \eqref{eq_main1} is obtained through an iterative process based on problem \eqref{eq_pde1} with a small parameter $\varepsilon$.
For a fixed $\varepsilon >0$, we begin with an initial guess $v_{1}$ and proceed by solving  the $n_{max}$ numerical  problems iteratively
\begin{equation}
G_{\varepsilon,h}^{v_{n}}(u_{h}^{n},x)=0, \quad n=1,\dots, n_{max}.
\label{gknp}
\end{equation}
At each iteration step $n$, given $v_{n}$, we compute the solution $u_{h}^{n}$. We then update the iteration
by setting $v_{n+1} \gets u_{h}^{n}$ 
and the problem is solved again using the 
new  value $v_{n+1}$.
After $n_{max}$ iterations, we obtain the numerical solution, which  is an approximate solution to
the problem
$$
G_{\varepsilon,h}^{u_h}(u_h(x),x)=0.
$$
To solve each  of the numerical problems in \eqref{gknp}, we discretize de second order derivatives as described 
at the beginning of Section \ref{sec_method}. This leads  to a nonlinear system given by 
$$
G_{\varepsilon,h}^{v_{n}}(u_{h,i}^{n},x_i)=0, \quad  x_i \in \Omega_h,
$$
where $u_{h,i}^{n}$ are the solutions at the interior points to be determined.
This operator also depends on   boundary points. However,  since their values are known, they are not treated as unknowns in the system.
To obtain the approximate solution over the discrete domain, the resulting nonlinear system of equations must be solved. 
For each fixed $v_n$, we apply 
an Euler map, that is, starting from an initial guess 
 $u_{h}^{n,1}$,  a number of  iterations are performed according to
$$
u_{h,i}^{n,m+1}  =u_{h,i}^{n,m} - \rho G_{\varepsilon,h}^{v_{n}}(u_{h,i}^{n,m},x_i), \quad m=1,\dots, m_{max}.
$$
Therefore for each $v_n$ we obtain the numerical solution $u_{h}^{n,m_{max}}$ which approximates the solution $u_h^n$.
 For clarity, the algorithm is presented below. 
 
 \begin{algorithm}[h]
\caption{Numerical scheme}
\begin{algorithmic}[]
\State \textbf{Input:} Set $\Omega$, the number of points $N$  and   the uniform grid size $h$.
\State \hspace{15pt} Set source term $f(x)$ 
\State \hspace{15pt} Set stability parameter $\rho =0.05 h^2$ and $\varepsilon = 1.2  h$
\State \hspace{15pt} Set number of outer iterations $n_{\max}$ and inner steps $m_{\max}$
\State \hspace{15pt}  Define nonlinear operators $F_1,F_2$ and $h_\varepsilon(v)$ to build $F_{\varepsilon,h}^v(M,u_h)$
\State \hspace{15pt} Initialize solution $v:=v_{1}$  and set $u_{h}^{1,1}:=v_1$

\For{$n = 1$ to $n_{\max}$}
    \For{$m = 1$ to $m_{\max}$}
            \State $u_h^{n,m+1} = u_{h}^{n,m} - \rho (F_{\varepsilon,h}^{v_n}(M,u_h^{n,m})-f)$
        \State Enforce boundary conditions of $u$
    \EndFor
     \State $v_{n+1}= u_h^{n,m_{max}}$
      \State $u_h^{n+1,1}= v_{n+1}$
\EndFor
\State \textbf{Output:} Numerical solution $u_h^{n_{max},m_{max}}$
\end{algorithmic}
\end{algorithm}


\begin{example}[One-dimensional problem]\label{ex_tp1}\normalfont 
We consider a one-dimensional problem defined in $[-1,1]$,
discretized using a uniform grid with  $N=250$ points. The  source term  is defined according to the exact solution of the problem which is
  $u_{exact}(x)=-x^2/2$ for $x<0$ and $u_{exact}(x)=x^2/2$ for $x>0$, see Figure \ref{fig1d}(a). 
  This profile corresponds to different concavities on either side 
of the origin.

\begin{figure}[h]
    \centering
    \begin{subfigure}[b]{0.49\textwidth}
        \includegraphics[width=\textwidth]{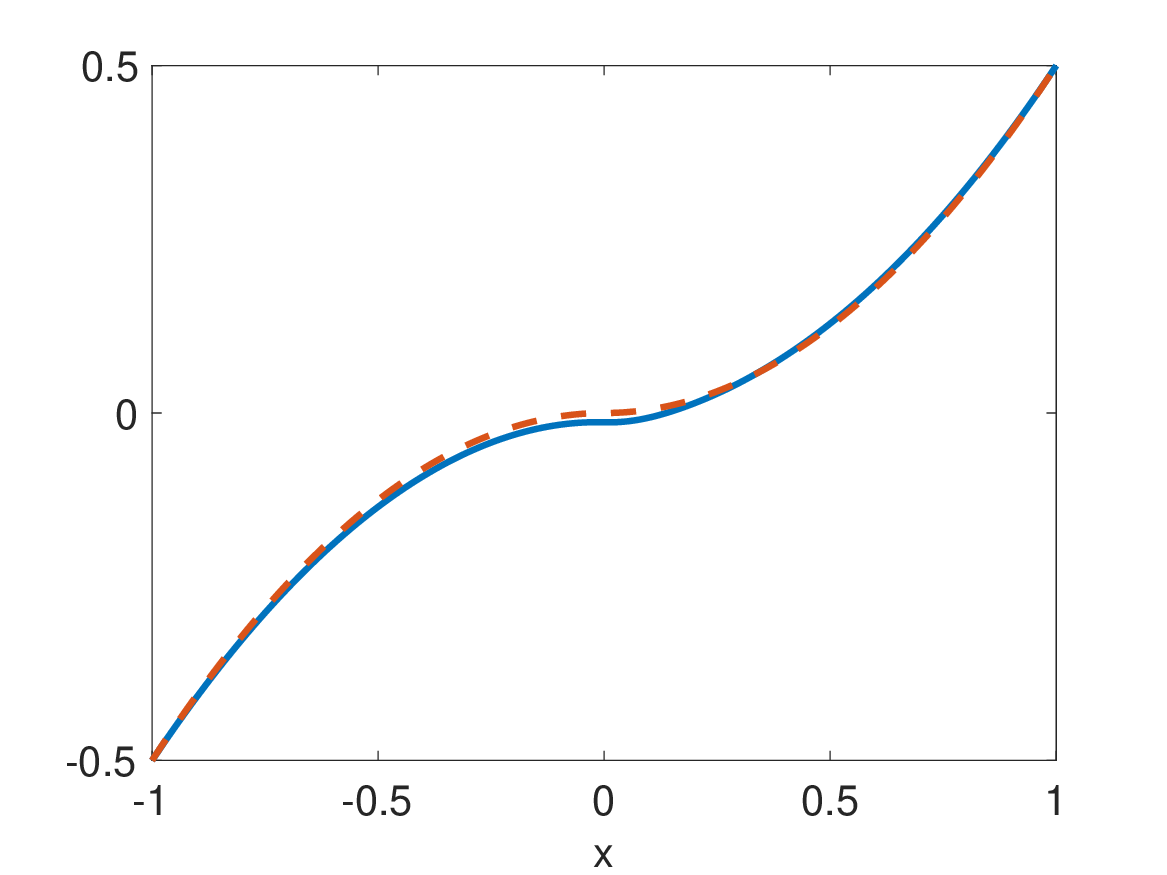}
         \caption{}
        \label{fig1da}
    \end{subfigure}
    \hfill
    \begin{subfigure}[b]{0.49\textwidth}
        \includegraphics[width=\textwidth]{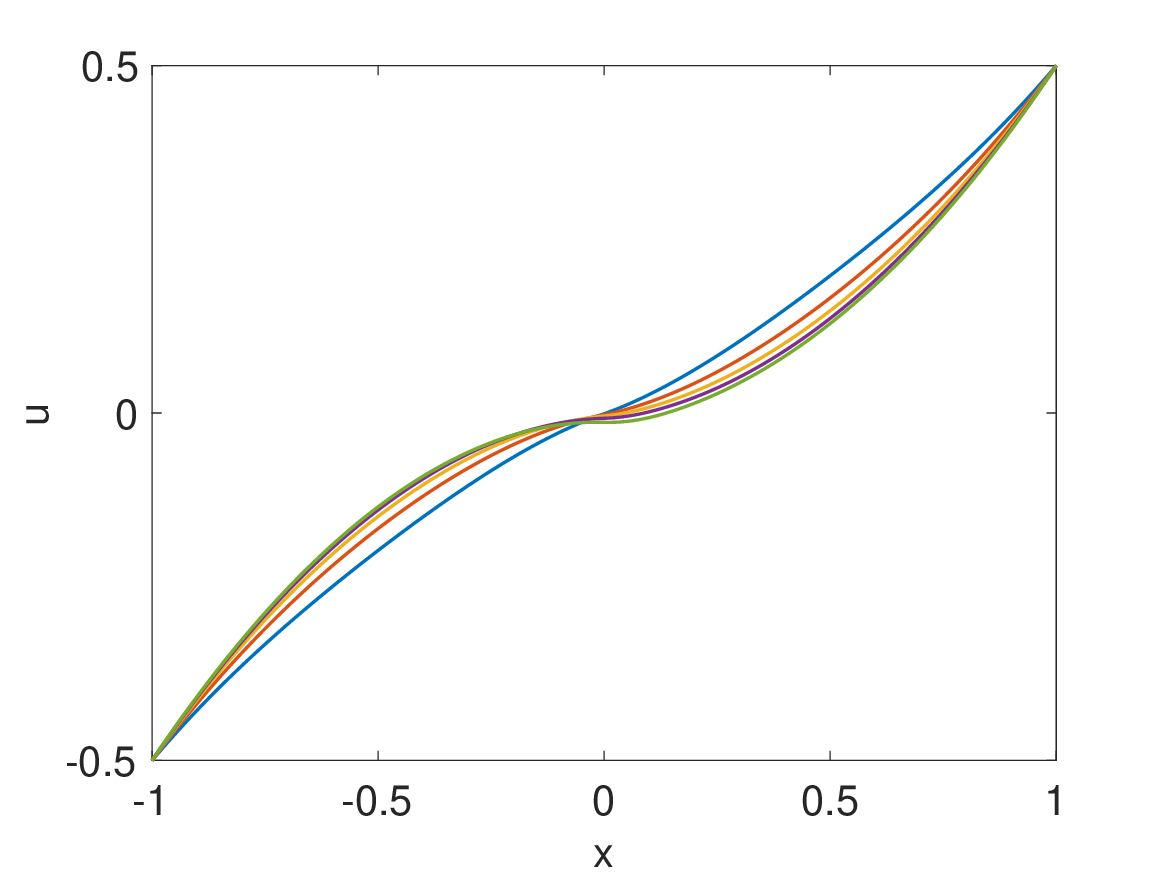}
         \caption{}
        \label{fig1db}
    \end{subfigure}
    \hfill
    \begin{subfigure}[b]{0.49\textwidth}
        \includegraphics[width=\textwidth]{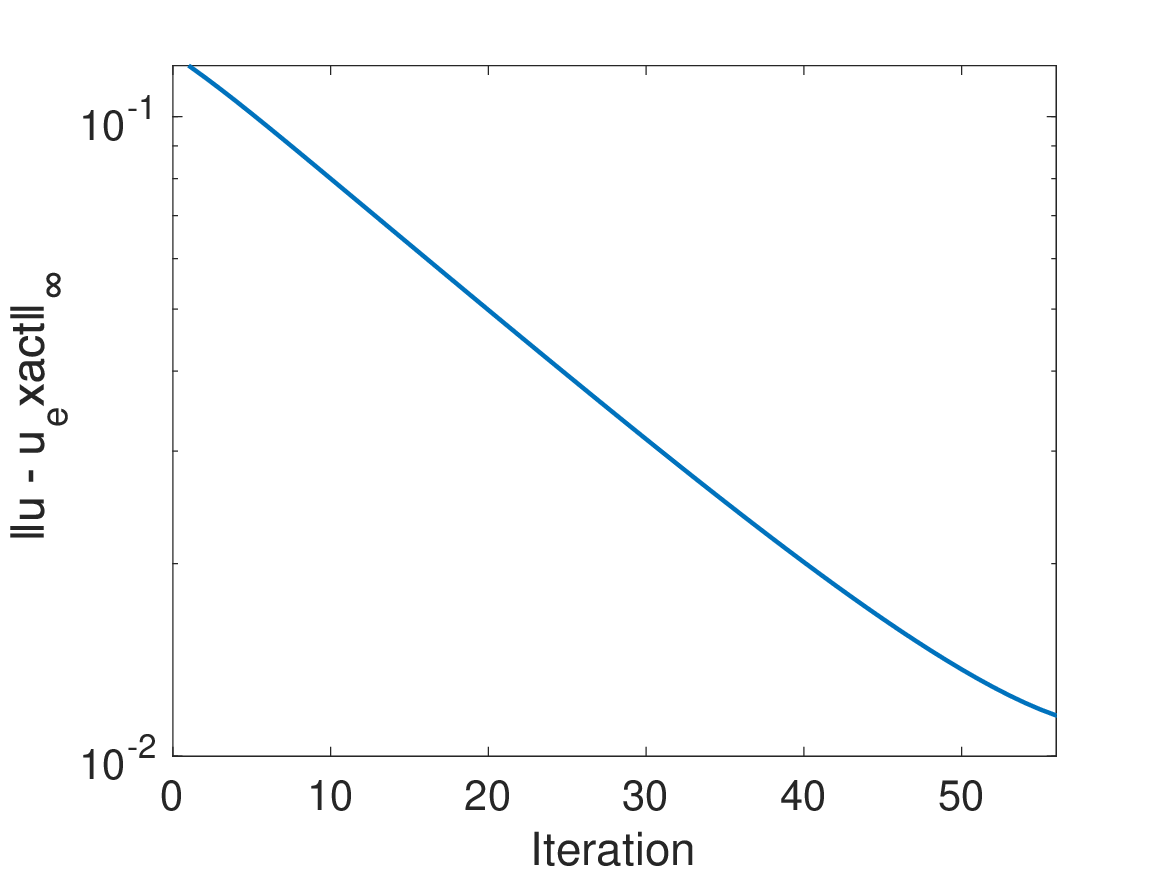}
         \caption{}
        \label{fig1dc}
    \end{subfigure}
    \caption{One-dimensional example: 
    (a) Exact solution $(- - )$ versus approximate solution $(\--)$; (b) Approximate solution $u_h^{n,m_{max}}$,  for  $n=1, \dots, 55$; 
    (c) Plot of the error $||u_h^{n,m_{max}}-u_{exact}||_\infty$ for  $n=1, \dots, n_{max}$. }
    \label{fig1d}
\end{figure}

To define the operator 
$F_{\varepsilon,h}^v(M,x_i)$ 
we define the two nonlinear operators $F_1$ and $F_2$ as
$$
F_1(M) = \max(-3M, -2 M), \qquad F_2(M) = \max(-M, -2 M).
$$

The numerical method consists in two iterative processes. An outer loop that updates the function $v$ and consequently the regularization function
$h_\varepsilon(v)$, and an inner loop that  applies explicit updates to solve the nonlinear discrete system.
This  second loop  mimics a pseudo-time-stepping procedure to reach a steady-state solution 
of the underlying nonlinear partial differential equation. See Figure \ref{fig1d}(b), to visualize a couple of iterations of the outer loop.

The value of $\varepsilon$ is set to $1.2 h$, where $h$ denotes the mesh size.
The second order derivative is approximated using the centered finite difference approximation and 
to ensure stability of the  explicit Euler method, which is used to iteratively solve the nonlinear system,
the parameter $\rho$ is choosen as
$\rho = 0.05 h^2$.  Dirichlet boundary conditions are imposed using the known values 
$u(-1)=-0.5$ and $u(1)=0.5$.

The algorithm tracks convergence by computing the maximum error with respect to the exact solution at each outer iteration
$||u_h^{n,m_{max}} -u_{exact}||_{\infty}, \ n=1,\dots, n_{max}$.
A semilogarithmic plot ($x$ is ploted in linear scale and $y$ is ploted in a logarithmic scale) illustrates the convergence behavior,
displayed in Figure \ref{fig1d}(c). 

\end{example}


\begin{figure}[ht]
    \centering
    \begin{subfigure}[b]{0.49\textwidth}
        \includegraphics[width=\textwidth]{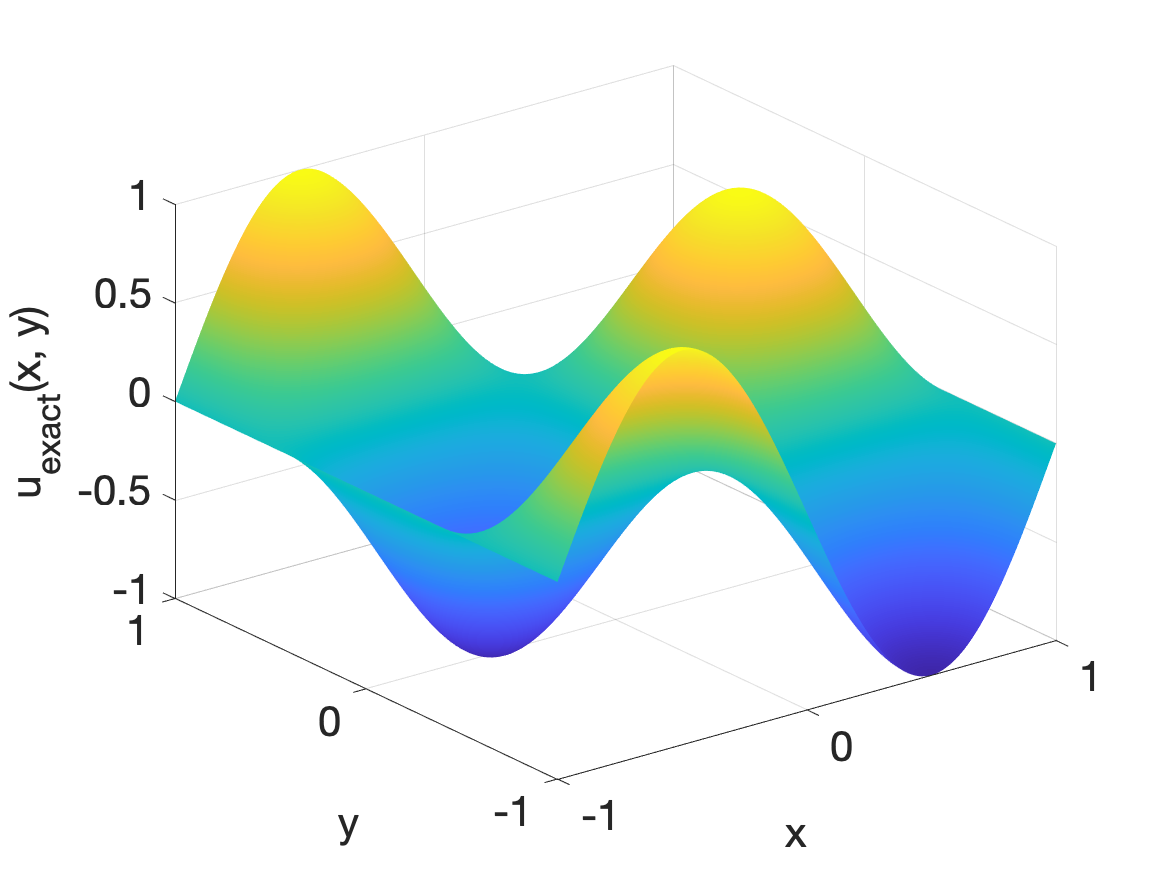}
        \caption{}
        \label{fig:exact}
    \end{subfigure}
    \hfill
    \begin{subfigure}[b]{0.49\textwidth}
        \includegraphics[width=\textwidth]{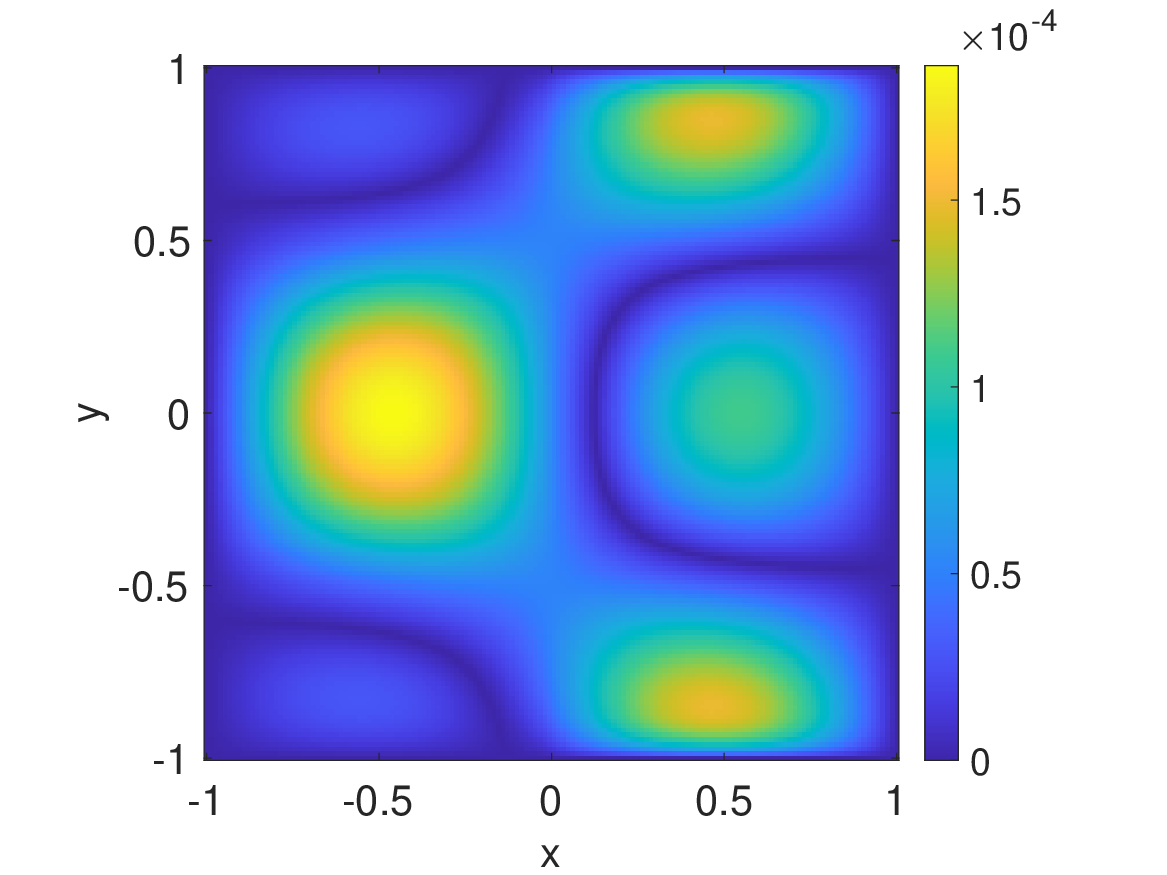}
            \caption{}
        \label{fig:heatmap}
    \end{subfigure}
    \hfill
    \begin{subfigure}[b]{0.49\textwidth}
        \includegraphics[width=\textwidth]{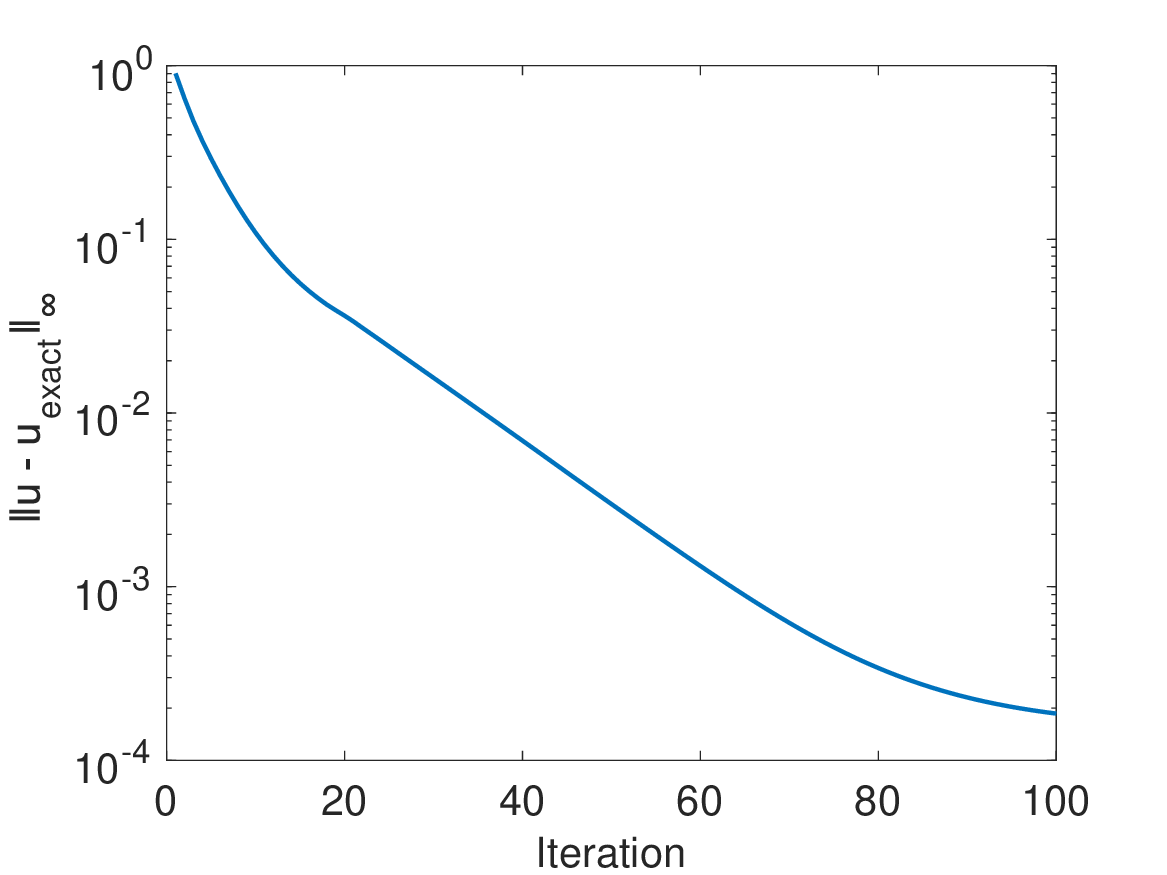}
            \caption{}
        \label{fig:convergence}
    \end{subfigure}
    \caption{Two-dimensional example:  (a) Exact solution; (b) Maximum error, that is, $||u^{n_{max},m_{max}}-u_{exact}||_\infty$, with $n_{max}=100$ ;
    (c) Error at each $v$ iteration,  that is, $||u^{n,m_{max}}-u_{exact}||_\infty$ for  $n=1, \dots, 100$.}
    \label{fig2d}
\end{figure}

\begin{example}[Two-dimensional problem]\label{ex_tp2}\normalfont 
For the two-dimensional problem the domain is $[-1,1]\times [-1,1]$ and we build a source term for the problem with the smooth exact solution
$u_{exact}(x,t) = \sin(\pi x) \cos(\pi y) $. The domain is discretized using $N=250$ points in each direction. The problem depends on $\varepsilon$ and this is chosen as 
$\varepsilon = 1.2 h$ where $h$ is  the mesh size in both directions.
The numerical method initializes with a smooth guess and iterates in two levels as described previously.
The outer iterations update the $v$ function and consequently the $h_\varepsilon(v)$.
The inner iterations apply explicit Euler updates to solve the nonlinear discrete system, for a fixed $v$.
Then the solution enters as the new $v$ and also as the new guess for the Euler iteration.

At each inner step, the boundary conditions are enforced using the  Dirichlet boundary conditions.
The second derivatives are approximated using central finite differences on the interior nodes. 
After each outer iteration, the maximum difference between the numerical and exact solution is recorded.

We show the exact solution, the convergence history, and the pointwise error at Figure \ref{fig2d}.

\end{example}

\vspace{.25in}

\noindent{\bf Acknowledgements - }The authors are supported by the Centre for Mathematics of the University of Coimbra (funded by the Portuguese Government through FCT/MCTES, DOI 10.54499/UIDB/00324/2020). This publication is based upon work supported by King Abdullah University of Science and Technology Research Funding (KRF) under Award No. ORFS-2024-CRG12-6430.3. 
\bigskip

\noindent{\bf Declarations}

\smallskip

\noindent {Data availability statement:} All data needed are contained in the manuscript.

\noindent {Funding and/or Conflicts of interests/Competing interests:} The authors declare that there are no financial, competing or conflicts of interest.

\bibliographystyle{plain}
\bibliography{biblio}

\smallskip

\bigskip

\noindent\textsc{Edgard A. Pimentel}\\
CMUC, Department of Mathematics\\ 
University of Coimbra\\
3001-143 Coimbra, Portugal\\
\noindent\texttt{edgard.pimentel@mat.uc.pt}

\bigskip

\noindent\textsc{Erc\'ilia Sousa}\\
CMUC, Department of Mathematics\\ 
University of Coimbra\\
3001-143 Coimbra, Portugal\\
\noindent\texttt{ecs@mat.uc.pt}

\end{document}